\documentclass[12pt,reqno]{amsart}
\usepackage[english]{babel}
\usepackage{hyperref}
\DeclareMathOperator{\ess}{ess}
\DeclareMathOperator{\arccot}{arccot}

\newtheorem{thm}{Theorem}

\newtheorem{lem}{Lemma}
\newtheorem{deff}{Definition}
\newtheorem{cor}{Corollary}
\newtheorem{conj}{Conjecture}
\newtheorem{cl}{Claim}

\theoremstyle{definition}

\newtheorem{example}{Example}
\newtheorem{rmk}{Remark}

\begin{document}

\title[
Asymptotic behavior of support points for planar curves]
{Asymptotic behavior of support points\\ for planar curves}

\author{Yu.G.~Nikonorov}

\address{Nikonorov Yuri\u{i} Gennadyevich
\newline
South Mathematical Institute
\newline
of  Vladikavkaz Scientific Center
\newline
of the Russian Academy of Sciences,
\newline Vladikavkaz, Markus str., 22,
\newline
362027, Russia}

\email{nikonorov2006@mail.ru}

\begin{abstract}
In this paper we prove a universal inequality describing the
asymptotic behavior of support points for planar continuous
curves. As corollaries we get an analogous result for tangent
points of differentiable planar curves and some (partially known)
assertions on the asymptotic of the mean value points for various
classical analytic theorems. Some open questions are formulated.

\vspace{2mm}


\vspace{2mm} \noindent {\it Key words and phrases:} parametric
continuous planar curves, smooth curves, mean-value theorems.

\end{abstract}

\maketitle

\section{Introduction and main results}

A starting point of this project is a remarkable property of
mean-value points in the first integral mean-value theorem. Let
$f:[0,1] \rightarrow \mathbb{R}$ be a continuous function. For any
$x\in (0,1]$ consider $\xi(x)$ that is the maximum of $\tau \in
[0,x]$ such that $x \cdot f(\tau)=\int_0^x f(t)\, dt$. Then the
inequality $\varlimsup\limits_{x\to 0} \frac{\xi(x)}{x} \geq
\frac{1}{e}$ holds. This inequality was proposed by Professor
V.K.~Ionin and was proved at first in the paper \cite{MV1}.
Further, this result was generalized in various ways \cite{MV2,
MV3, MV4, MV5, MV6, MV7}. But in this paper we suggest another
point of view: the main object of our study are not functions, but
continuous curves (for example, the graphs of functions).

At first we refine a definition of support points (with respect to
a given chord) of a continuous parametric curve in spite of the
fact that this notion is quite natural and intuitively clear.

\begin{deff}\label{1}
Let $\gamma:(a,b) \rightarrow \mathbb{E}^2$ be a continuous
parametric curve in the Euclidean plane, $[c,d] \subset (a,b)$. We
say that a point $\gamma(\tau_0)$, $\tau_0\in [c,d]$ is a support
point for the chord $[\gamma(c),\gamma(d)]$ (if
$\gamma(c)\neq\gamma(d)$), if a straight line $l$ passing through
$\gamma(\tau_0)$ in parallel $[\gamma(c),\gamma(d)]$ is such that
for all $\tau$, rather close to $\tau_0$, the points
$\gamma(\tau)$ are in one and the same half-plane determined by
$l$. If $\gamma(c)=\gamma(d)$, then we set that any point
$\gamma(\tau_0)$ for $\tau_0\in [c,d]$ is a support point for the
(degenerate) chord $[\gamma(c),\gamma(d)]$.
\end{deff}

Note that our convention on the set of support points for
$\gamma(c)=\gamma(d)$ is stipulated by the universality of an
analytic description of such sets under this definition (see
below). It is possible to modify this definition, but in any case
it is not so important since the case of chords with zero length
is trivial (in some sense) for our questions.

Consider a (rectangular) Cartesian coordinate system $Oxy$ in the
plane $\mathbb{E}^2$. Then $\gamma(t)=(x(t),y(t))\in
\mathbb{R}^2$, $t\in (a,b)$. The fact that a point
$\gamma(\tau_0)$, $\tau_0\in [c,d]$, is a support point for the
chord $[\gamma(c),\gamma(d)]$ could be expressed in the following
form. Consider a function $\Phi:(a,b) \rightarrow \mathbb{R}$,
$$
\Phi(\tau)=\det \left( \begin{array}{cc}
x(d)-x(c) & y(d)-y(c) \\
x(\tau) & y(\tau) \\
\end{array}
\right).
$$
It is easy to see that a point $\gamma(\tau_0)$, $\tau_0\in
[c,d]$, is a support point for the chord $[\gamma(c),\gamma(d)]$
if and only if $\tau_0$ is a point of local extremum of the
function $\Phi$. Since $\Phi(c)=\Phi(d)$, then there is at least
one point of local extremum of the function under consideration on
the interval $(c,d)$. Moreover, if the curve $\gamma(t)$ is
differentiable at the point $\tau_0$, then
$$
\Phi^{\prime}(\tau_0)=\det \left( \begin{array}{cc}
x(d)-x(c) & y(d)-y(c) \\
x^{\prime}(\tau_0) & y^{\prime}(\tau_0) \\
\end{array}
\right)=0,
$$
that means the collinearity of the vectors  $\gamma(d)-\gamma(c)$
and $\gamma^{\prime}(\tau_0)$. Note also that the straight line
$l$ in Definition \ref{1} is a tangent line to the curve
$\gamma(t)$ at the point $\gamma(\tau_0)$ when
$\gamma(c)\neq\gamma(d)$ and $\gamma^{\prime}(\tau_0)\neq 0$.

Now we can formulate the main results of this paper.
Let $\gamma: [a,b) \rightarrow \mathbb{E}^2$, where $a,b \in
\overline{\mathbb{R}}=\mathbb{R}\cup \{-\infty,\infty\}$, be a
continuous parametric curve in the Euclidean plane, that is not a
constant in any neighborhood of the point $a$.

Further, by $D(t)$ we denote a distance between the points
$\gamma(a)$ and $\gamma(t)$.

For every $t\in (a,b)$ we denote by $S(t)$ a set of $\tau \in
(a,t]$ such that the point $\gamma(\tau)$ is a support point for
the chord $[\gamma(a),\gamma(t)]$. Now, consider
$$
DS(t)=\sup \{D(\tau)\,|\,\tau\in S(t)\}.
$$

The main object of our study is  the asymptotic of the ratio
$DS(t)/D(t)$ when $t \rightarrow a$. For a fixed $t$ the set
$S(t)\subset (a,t]$ could be organized quite complicated, and this
is evident from a geometric interpretation of this set. The case
$D(t)=0$ is exceptional. According to the definition, in this case
we get $S(t)=(a,t]$. Obviously, there exists $\tau \in S(t)=(a,t]$
with the property $D(\tau)>0$ (otherwise, the curve $\gamma$ is
constant on the interval $(a,t)$). Hence $DS(t)>0$ for a such $t$,
and we set $DS(t)/D(t)=\infty$ when $D(t)=0$.

For a fixed value $t$ it is possible to choose a curve with the
ratio $DS(t)/D(t)$ equal to a given positive number. On the other
hand, it is clear that this ratio could not be greater than $1$
for all values of a parameter. A rather less evident fact is that
the ratio $DS(t)/D(t)$ could not be less than some definite
positive number for all values of a parameter. An exact assertion
consists in the following theorem that is the main result of this
paper.

\begin{thm}\label{main0}
Let $\gamma: [a,b) \rightarrow \mathbb{E}^2$ be an arbitrary
continuous parametric curve. Then the inequality
\begin{equation}\label{eq0}
\varlimsup_{t\to a} \frac{DS(t)}{D(t)}\geq \frac{1}{e}
\end{equation}
holds, where, as usual, $e =\lim\limits_{n\to \infty} \left( 1
+\frac{1}{n}\right)^n$.
\end{thm}

For a differentiable curve any support point is a tangent point
automatically. Therefore, the above theorem implies some
corresponding results for tangent points. Let us clarify the
statement of the problem.

Let $\gamma: [a,b) \rightarrow \mathbb{E}^2$, where $a,b \in
\overline{\mathbb{R}}=\mathbb{R}\cup \{-\infty,\infty\}$, be a
continuous parametric curve in the Euclidean plane such that for
every $t\in (a,b)$ there exists a non-zero derivative vector
$\gamma^{\prime}(t)$. Note that this vector defines a direction
of the tangent line to the considered curve at the point
$\gamma(t)$. If the derivative vector is continuous (with respect
to a parameter), then such a curve is called smooth regular,
but we do not assume the continuity of the derivative vector
in what follows (unless otherwise stipulated).

By analogy with general continuous curves, for every $t\in (a,b)$
we denote by $T(t)$ the set of $\tau \in (a,t]$ such that the
vector $\gamma^{\prime}(\tau)$ is collinear  to the vector
$\overrightarrow{\gamma(a)\gamma(t)}$. It is clear that the set
$T(t)$ is non-empty for every $t\in (a,b)$ (since even the set
$S(t)\subset T(t)$ is non-empty).

Let us consider the value
$$
DT(t)=\sup \{D(\tau)\,|\,\tau\in T(t)\}.
$$

We are interested in the asymptotic of the ratio
$DT(t)/D(t)$ when $t \rightarrow a$.  For a fixed $t$ (by analogy
with $S(t)$) the set $T(t)\subset (a,t]$ could be rather
complicated that follows from the geometric interpretation of this
set as a set of points $\tau \in (a,t]$, such that a tangent line
to the curve $\gamma$ at the point $\gamma(\tau)$ is parallel to
the chord $[\gamma(a), \gamma(t)]$. For example, for $D(t)=0$
we get $T(t)=(a,t]$, and we set $DT(t)/D(t)=\infty$ in this case.
This is motivated by the fact that for some $\tau \in (a,t)\subset
T(t)$ the inequality $D(\tau)>0$ holds (otherwise the derivative
vector $\gamma^{\prime}$ should be trivial on the interval
$(a,t)$) and, therefore, $DT(t)>0$ (compare with analogous
convention for $DS(t)/D(t)$ when $D(t)=0$).

It is clear that for some fixed value of $t$ (and for a curve
chosen specially) the ratio $DT(t)/D(t)$ could be equal to any
positive number. For a differentiable curve $\gamma$ we obviously
get the inclusion $S(t)\subset T(t)$ (but $S(t) \neq \emptyset$
under the condition) and, therefore, the inequality $DS(t) \leq
DT(t)$ for an arbitrary $t\in (a,b)$. Hence, the following result
(obtained at first in \cite{NikTang}) is an immediate consequence
of Theorem \ref{main0}.

\begin{thm}[\cite{NikTang}]\label{main1}
Let $\gamma: [a,b) \rightarrow \mathbb{E}^2$ be an arbitrary
continuous parametric curve with non-zero derivative vector
$\gamma^{\prime}(t)$ at every point $t\in (a,b)$. Then the
inequality
\begin{equation}\label{eq1}
\varlimsup_{t\to a} \frac{DT(t)}{D(t)}\geq \frac{1}{e}
\end{equation}
holds.
\end{thm}

In Section 2 we consider some examples illustrating the assertions
of Theorem \ref{main0} and Theorem \ref{main1}. According to these examples, the
inequalities (\ref{eq0}) and (\ref{eq1}) are best possible.

Note also that the inequalities (\ref{eq0}) and (\ref{eq1}) have
local character. Therefore,
the domain of definition $[a,b)$ of the curve $\gamma(t)$
in Theorem \ref{main0} or in Theorem \ref{main1} can be replaced
by any interval $[a,b_1)$, where $b_1 \in (a,b)$.

Under conditions of Theorem \ref{main0}, we can consider curve $\gamma_1:
[a_1,b_1) \rightarrow \mathbb{E}^2$ instead of the curve $\gamma:
[a,b) \rightarrow \mathbb{E}^2$, if $\gamma_1(t)=\gamma(g(t))$ for
some continuous bijective function $g:[a_1,b_1)\rightarrow [a,b)$.
In the case of Theorem \ref{main1}, it should be required (in addition) the
existence of positive derivative for the function $g$ on the
interval $(a_1,b_1)$. In other words, the assertions
of the above theorems  concern the geometry
of a nonparametric curve (that could be considered as a class of pairwise equivalent
parametric curves).
Further (e.~g., in the proof of Theorem
\ref{main0} in Section 3) we will use these properties repeatedly.

On the other hand, the results of Theorem \ref{main0} and Theorem
\ref{main1} may be used to study some special parameterizations
of a curve. In such a case we obtain a couple of (partially
known) results on the asymptotic of mean value points in some
classical differential and integral theorems (cf. Section 4). In
the last section we formulate some unsolved questions that can be used as a basis
for further investigations in the designated
direction. In particular, it would be desirable to hope that
results of this paper initiate more detailed study of
the asymptotic behaviour of the ratios $DS(t)/D(t)$ and $DT(t)/D(t)$.

For functions defined on an interval
$[\alpha, \beta] \subset \mathbb{R}$, we use limits and
derivatives at the point $\alpha$ ($\beta$, respectively) having in mind
right hand limits  and right derivatives (left hand limits and
left  derivatives, respectively). This natural convention allow
us to simplify the presentation.

\section{Some examples}

We will define curves with using of polar coordinates with a pole
in the initial point of the curve. Moreover, a parameter in the
following examples is the polar angle $t$, $t\in [a,b)$. A curve
$\gamma(t)$, $t\in [a,b)$ is defined completely by $t\mapsto
\rho(t)$ -- the distance function from the current point of the
curve to the pole ($\rho(a)=0$). The points $\tau \in
T(t)\subset(a,t]$ are determined by the equation
\begin{equation}\label{eq5}
\rho(\tau) \cos(t-\tau)-\rho^{\prime}(\tau) \sin(t-\tau)=0.
\end{equation}
Note also that the equality $D(t)=\rho(t)$ holds in all examples
below. Moreover, in all these examples, the equality $T(t)=S(t)$
holds for all possible values of the parameter $t$ (it is easy to check),
i.~e. all tangent points are also support
points for the corresponding chords of the curve.

\begin{example}\label{ex1}
Consider a curve defined by the equation $\rho(t)=e^{\alpha t}$,
where $\alpha >0$, the parameter is the polar angle, $t\in
[-\infty,b)$. It is clear that $t=\tau+\arccot(\alpha)+\pi n$ ($n
\in \mathbb{Z}$, $n \geq 0$) for $\tau \in T(t)=S(t)$. Hence,
$$
DS(t)=DT(t)=e^{\alpha t-\alpha\cdot\arccot(\alpha)}, \quad D(t)=\rho(t)=e^{\alpha t},
\quad \frac{DT(t)}{D(t)}=\frac{DS(t)}{D(t)}=e^{-\alpha\cdot \arccot(\alpha)}.
$$
It is necessary to note that $\alpha\cdot \arccot(\alpha) <1$ for
$\alpha >0$ and $\lim\limits_{\alpha \to \infty} \bigl(
\alpha\cdot \arccot(\alpha) \bigr)=1$. Therefore, the inequalities
in Theorem \ref{main0} and Theorem \ref{main1} are best possible.
\end{example}

\begin{example}\label{ex2}
Consider a curve defined by the equation $\rho(t)=t^{\alpha}$,
where $\alpha >0$, the parameter is the polar angle, $t\in [0,b)$.
For determining of $\tau \in T(t)=S(t)$ we have the equation
$$
\tau \cos(t-\tau)-\alpha \sin(t-\tau)=0.
$$
It is clear that $t=\tau+\arccot \left( \frac{\alpha}{\tau}
\right)+\pi n$ ($n \in \mathbb{Z}$, $n \geq 0$). Therefore,
$$
\lim\limits_{t \to 0}\frac{DT(t)}{D(t)}= \lim\limits_{t \to
0}\frac{DS(t)}{D(t)}=\lim\limits_{\tau \to 0}\frac{\tau^{\alpha}}{\left(\tau+\arccot
\left( \frac{\alpha}{\tau} \right) \right)^{\alpha}}=
$$ $$
\lim\limits_{\tau \to 0}\left( 1+\frac{1}{\tau} \arccot \left(
\frac{\alpha}{\tau}\right)\right)^{-\alpha}=\left(1 +\frac{1}{\alpha}\right)^{-\alpha}
>e^{-1}.
$$
Note also that $\lim\limits_{\alpha \to \infty} \left(1
+{\alpha}^{-1}\right)^{-\alpha}=e^{-1}$.
\end{example}

\begin{example}\label{ex3}
Let us set $l>1$, $[a,b)=[-\infty, 0)$, $\rho(t)=e^{\phi(t)}$,
where $\phi(t)=-|t|^l=-(-t)^l$, a parameter is the polar angle
again, $t\in [-\infty,0)$. Obviously,
$t=\tau+\arccot(\phi^{\prime}(\tau))+\pi n$ ($n \in \mathbb{Z}$,
$n \geq 0$) for $\tau \in T(t)=S(t)$ (cf. the equality
(\ref{eq5})). This implies immediately
$$
\lim\limits_{t \to -\infty} \frac{DT(t)}{D(t)}=\lim\limits_{t \to -\infty}
\frac{DS(t)}{D(t)}=e^L,
$$
where $L=\lim\limits_{s \to -\infty} \bigl(
\phi(s)-\phi(s+\alpha(s))\bigr)$ and
$\alpha(s)=\arccot(\phi^{\prime}(s))$. Since $\phi^{\prime
\prime}(t)\leq 0$, then
$$
\phi^{\prime}(s)\cdot \alpha(s) \geq \phi(s+\alpha(s))-\phi(s)\geq
\phi^{\prime}(s+\alpha(s))\cdot \alpha(s).
$$
Since
$$
\lim\limits_{\beta \to \infty} \bigl( \beta\cdot \arccot(\beta) \bigr)=1, \quad
\lim\limits_{s \to -\infty} \phi^{\prime}(s)=\infty, \quad \lim\limits_{s\to -\infty}
\frac{\phi^{\prime}(s+\alpha(s))}{\phi^{\prime}(s)}=1,
$$
then $L=-1$ and $\lim\limits_{t \to -\infty}
\frac{DT(t)}{D(t)}=\lim\limits_{t \to -\infty}
\frac{DS(t)}{D(t)}=e^{-1}$. This example helps to understand
better some steps in the proof of Theorem \ref{main0}.
\end{example}

A couple of other examples follows from Theorem \ref{main4} and
from Corollary \ref{cor1} due to the asymptotic equation
(\ref{odno.1}), because such examples are considered in various
papers, devoted to the asymptotic of mean value points in
classical mean value theorem \cite{MV1, MV2, MV3, MV4}. Under some
additional restrictions to the asymptotic of the curve $t
\rightarrow \gamma (t)$ at the point $a$, there exists a usual
limit $\lim\limits_{t\to a} \frac{DT(t)}{D(t)}\,\bigl( \geq e^{-1}
\bigr)$. Assertions of such kind for various integral and
differential mean value theorems are obtained in the papers
\cite{Abel, AbIv, Adik, Azpet, Jac, Mera, PoRiSa, Trif, Zhang, XCH}. Moreover, this
problematic is adequately depicted in the book \cite{SaRi}, where
one can find also extensive references.

\section{Proof of Theorem \ref{main0}}

Let us consider a Cartesian coordinate system $Oxy$ in
$\mathbb{E}^2$ such that $O=\gamma (a)$. Then
$\gamma(t)=(x(t),y(t))\in \mathbb{R}^2$, $t\in [a,b)$, and
$\gamma(a)=(x(a),y(a))=(0,0)$. The fact, that a point $\tau_0\in
(a,t]$ is in the set $S(t)$, can be expressed in the following
form. Consider a function $\Phi:[a,b) \rightarrow \mathbb{R}$,
\begin{equation}\label{phi}
\Phi(\tau)=\det \left(
\begin{array}{cc}
x(t) & y(t) \\
x(\tau) & y(\tau) \\
\end{array}
\right). \end{equation} Then a point $\tau_0\in (a,t]$ is in the
set $S(t)$ if and only if $\tau_0$ is a point of local extremum of
the function $\tau \mapsto \Phi(\tau)$.

In the rest of this section we prove Theorem \ref{main0}. {\it
Further, we suppose that the assertion of Theorem \ref{main0} does
not hold, and get the contradiction.}
\smallskip

Without loss of generality we may assume that $\gamma(t)\neq
\gamma(a)$ for all $t\in (a,b)$. Indeed, if there is a sequence
of points $t_n \in (a,b)$ such that $t_n\to a$ as $n \to \infty$
and $\gamma(t_n)=\gamma(a)$, then
$\varlimsup\limits_{t\to a} \frac{DT(t)}{D(t)}= \infty>\frac{1}{e}$
(in this case $D(t_n)=0$ and
$DS(t_n)/D(t_n)=\infty$ according to our arrangements discussed just
before the statement of Theorem \ref{main0}), that is impossible.
Therefore, numbers $t\in (a,b)$
with the property $\gamma(t)=\gamma(a)$ cannot be
close to $a$ as much as possible. Therefore, if we decrease (if necessary)
the number $b$, then we get that such points $t$ are absent.

Let us consider functions $\rho, \theta:[a,b) \rightarrow
\mathbb{R}$, defined in the following way. Put $\rho(t)=D(t)$ --
the distance between $O$ and a current point of the curve
$\gamma(t)$. As $\theta(t)$ we consider a number satisfied to
equations $x(t)=\rho(t)\cos(\theta(t))$ and $y(t)=\rho(t)\sin(\theta(t))$.
Such a number (the polar angle) is defined uniquely up to a summand
$2\pi n$ ($n\in \mathbb{Z}$). Taking into account the continuity
of $\gamma(t)$, it is easy to choose this angle
in such a way that the function $t \mapsto \theta(t)$ is continuous for all values
of $t$.

Let us show that we may assume $\theta(t)$ to be strictly
increasing. In our new notations, the function $\Phi$ (cf. the
equality (\ref{phi})) has the following form:
$$
\Phi(\tau)=\rho(t)\rho(\tau)\sin\Bigl(\theta(\tau)-\theta(t)\Bigr).
$$
If $t$ is a point of local maximum (minimum) of the function $\tau
\mapsto \theta (\tau)$, then it is also a point of local minimum
(maximum, respectively) of the function $\tau \mapsto \Phi(\tau)$,
$t \in S(t)$, and $DS(t)\geq D(t)$. Therefore (cf. reasonings
above), such points can not be close to $a$ as much as posiible. Decreasing
(if necessary) the number $b$, we may assume  that
the point $t$ is not a point of local extremum of the
function $\tau \mapsto \theta (\tau)$ for all $t \in
(a,b)$. Taking into account the
continuity of this function, we get that it is either strictly
decreasing or strictly increasing on the interval $(a,b)$.
Replacing (if necessary) the ray $Oy$ with the opposite ray
(hence, changing the orientation), we may assume that this
function is strictly increasing for $t\in (a,b)$.

Now, we may change (without loss of generality) the parameter
$t$ in such a way that $\theta(t)=t$ for all $t\in (a,b)$, i.~e.
the curve under consideration is parameterized by the polar angle.
Further, it will be convenient to consider a function
$$
\phi(t)=\ln (\rho(t)).
$$
Since $\rho(t)=e^{\phi(t)}$, then (according to our conclusions
and assumption on $\rho(t)$) the function $\phi:[a,b) \rightarrow
\mathbb{\overline{R}}$ is continuous, takes finite values for
$t\in (a,b)$, and $\phi(a)=-\infty$. Further, we determine some other
properties of this function.

Since we supposed Theorem \ref{main0} to be false, then we may
assume that there is a number $q>1$ such that
$$
\frac{D(\tau)}{D(t)}\leq e^{-q}
$$
for all $t\in(a,b)$ and all $\tau \in S(t)$.
In our notations $D(t)=e^{\phi(t)}$, hence this inequality
is equivalent to the following one:
\begin{equation}\label{eq3}
\phi(t)-\phi(\tau) \geq q>1
\end{equation}
for all $t\in (a,b)$ and for all $\tau\in S(t)$.

Now, since $x(t)=e^{\phi(t)}\cos(t)$ and
$y(t)=e^{\phi(t)}\sin(t)$, then  a point $\tau \in (a,t]$ is in
the set $S(t)$ if and only if $\tau$ is a point of local extremum
of the function
$$
\tau \mapsto -(\rho(t))^{-1}
\Phi(\tau)=\rho(\tau)\sin(t-\theta)=e^{\phi(\tau)}\sin(t-\tau).
$$
For a fixed $t$ we consider an interval $I(t)=[\max\{t-\pi,
a\},t]$. It is clear that the function $\tau \mapsto
e^{\phi(\tau)}\sin(t-\tau)$ is vanished at the endpoints of this
interval. Therefore, there is at least one point of extremum of
the latter function (i.~e. a point in the set $S(t)$) in the
interior of this interval.

Further, we consider the function
$$
\tau \mapsto \phi(\tau)+\ln
(\sin(t-\tau))=\ln(e^{\phi(\tau)}\sin(t-\tau))=:F^t(\tau).
$$

\begin{cl}\label{cl1} For every $t\in(a,b)$, there is $\beta(t)>0$ such that the function
$\tau \mapsto F^t(\tau)$ is strictly decreasing on the interval
$[t-\beta(t), t]$.
\end{cl}

\begin{proof}
Suppose the contrary. Then there are sequences of numbers
$\{\tau_n\}$ and $\{\xi_n\}$ such that $\tau_n <\xi_n <t$ and
$F^t(\tau_n)\leq F^t(\xi_n)$ for all $n$, $\tau_n \to t$ as $n\to
\infty$. Since $F^t(\tau) \to -\infty$ as $\tau \to t-0$, then
there is a point $\eta_n$ of absolute maximum of the function
$F^t$ on the interval $[\tau_n, t)$. Clear, that $\eta_n\in S(t)$
and $\eta_n \to t$ as $n\to \infty$. But according to our
assumption,  the inequality $\phi(t)-\phi(\eta_n)>q$
(the inequality (\ref{eq3})) holds for all $n$, that is impossible
(it suffices to pass to the limit in this inequality). Therefore, we
have proved the existence of the required $\beta(t)>0$ (it is
easy to see also that $\beta(t)<\pi$).
\end{proof}

\begin{rmk}\label{re2}
Recall that every increasing function $f:[\alpha,\beta]\subset
\mathbb{R} \rightarrow \mathbb{R}$ is differentiable almost
everywhere, its derivative is non-negative and summable, and
$\int\limits_{\alpha}^{\beta} f^{\prime}(t) dt \leq f(\beta) -
f(\alpha)$. Moreover, if  $f$ has derivative {\it at every point} of
the interval $[\alpha, \beta]$, then the above inequality becomes
an equality (in this case the function $x\mapsto f(x)$ is
absolutely continuous on the interval $[\alpha,\beta]$), cf.
\cite{Nat, RiS}.
\end{rmk}

\begin{cl}\label{cl2} For any $t\in(a,b)$ the function $\phi$ is differentiable almost everywhere
on every interval $[c,d] \subset (t-\beta(t), t)$, its derivative
is summable and
$$
\phi(d)-\phi(c) \leq \int\limits_c^d \phi^{\prime} (\tau) d\tau.
$$
\end{cl}

\begin{proof}
According to Claim \ref{cl1} the function $\tau \mapsto
F^t(\tau)=\phi(\tau)+\ln (\sin(t-\tau))$ decreases on the interval
$[t-\beta(t), t]$. Therefore, the function $-F^t$ increases on
this interval. Using properties of increasing functions (cf.
Remark \ref{re2}), differentiability and absolute continuity of
the function $\tau \mapsto \ln (\sin(t-\tau))$ on the interval
$[c,d]\subset (t-\beta(t), t)\subset (t-\pi, t)$, we easily get
the required properties of the function $\phi$.
\end{proof}

\begin{cl}\label{cl3} The function $\phi$ is differentiable almost everywhere on the interval $(a,b)$.
Its derivative $\phi^{\prime}$ is summable on every interval
$[c,d] \subset (a,b)$ and satisfies the inequality
$$
\phi(d)-\phi(c) \leq \int\limits_c^d \phi^{\prime} (\tau) d\tau.
$$
\end{cl}

\begin{proof}
For every $t\in (a,b)$ we consider the interval
$I(t):=(t-\beta(t),t)$ (Claim \ref{cl1}). All these intervals
cover jointly the interval $[c,d]$. By compactness,
$[c,d]$ is covered also by some finite subset of the intervals
$I(t)$, say, by  $I(t_0), I(t_1), \dots, I(t_l)$, $t_0<t_1 <
\cdots <t_l$. Now, choose numbers $s_i$, $i=0,\dots, l$, such that
$c=s_0<s_1<\cdots <s_{l-1}=d$ and $[s_i,s_{i+1}]\subset I(t_i)$.
According to Claim \ref{cl2} the function $\phi$ is differentiable
almost everywhere on every interval $[s_i,s_{i+1}]$, and the
inequality
$$
\phi(s_{i+1})-\phi(s_i) \leq \int_{s_i}^{s_{i+1}} \phi^{\prime}
(\tau) d\tau
$$
holds. Hence, $\phi$ is differentiable almost everywhere on the
interval $[c,d]$. Summing the obtained inequalities by $i$ from
$0$ to $s-1$, we get an analogous inequality on the interval
$[c,d]$. Since the interval $[c,d] \subset (a,b)$ is arbitrary,
the function $\phi$ is differentiable almost everywhere on the
interval $(a,b)$.
\end{proof}

Further, it will be helpful to consider the set
$$
Sm =\{t \in (a,b)\,|\, \mbox{there exists}\, \phi^{\prime}(t) \in \mathbb{R} \}.
$$
Consider also the function $\alpha: Sm\rightarrow \mathbb{R}$,
defined by the equation
$$
\alpha(t)=\arccot(\phi^{\prime}(t)).
$$
It is clear that $\alpha(t)\in (0,\pi)$ for all values of the
parameter.

\begin{cl}\label{cl4} For every $\tau \in Sm$ either the inequality
$\tau+\alpha(\tau)\geq b$, or the inequality
$\phi(\tau+\alpha(\tau))-\phi(\tau)\geq q$ holds.
\end{cl}

\begin{proof}
Let us fix some $\tau_0 \in Sm$ and suppose that
$t_0:=\tau_0+\alpha(\tau_0)<b$. If the point $\tau_0$ is a point
of local extremum of the function $\tau \mapsto
e^{\phi(\tau)}\sin(t_0-\tau)$, then $\tau_0 \in S(t_0)$ and
(according to our assumptions) the inequality
$\phi(t_0)-\phi(\tau_0)> q$ (the inequality (\ref{eq3})) holds,
that implies the required result. However, $\tau_0$ should not
be a point of local extremum of the above function, but in any case, the point
$\tau_0$ is a critical point of the function $\tau \mapsto e^{\phi(\tau)}\sin(t_0-\tau)$
($\alpha(\tau_0)=\arccot(\phi^{\prime}(\tau_0))$ by definition).
In other word, the tangent line to
the curve $\gamma(t)$ at the point $\gamma(\tau_0)$ is parallel to
the chord $[O=\gamma(a),\gamma(t_0)]$.

Now, choose sequences of numbers $\{\tau_n\}$ and $\{t_n\}$ such
that $\tau_n \to \tau_0$, $t_n \to t_0$ as $n \to \infty$ and the
chord $[O, \gamma(t_n)]$ is parallel to the chord
$[\gamma(\tau_0),\gamma(\tau_n)]$ for all $n$. Since $\tau_0 \in Sm$, then
$$
\frac{1}{\tau_n-\tau_0}
\overrightarrow{\gamma(\tau_n)\gamma(\tau_0)}\to
\gamma^{\prime}(\tau_0) \,\mbox{    as   }\, n \to \infty.
$$
Let us show that for every $n$ there exists a number $\eta_n\in S(t_n)$
between the numbers $\tau_0$ and $\tau_n$. Such a number should be
a point of local extremum of the function $\tau \mapsto
e^{\phi(\tau)}\sin(t_n-\tau)$. For this goal we consider the
function $\Psi(\tau)=\det \left(
\begin{array}{c}
\gamma(\tau_n)-\gamma(\tau_0) \\
\gamma(\tau) \\
\end{array}
\right)$. Since $\Psi(\tau_0)=\Psi(\tau_n)$, then there is a point
$\eta_n$ of local extremum of this function between the points
$\tau_0$ and $\tau_n$. But the same point is also a point of local
extremum of the function
$$
\tau \mapsto \det \left(
\begin{array}{c}
\gamma(t_n) \\
\gamma(\tau) \\
\end{array}
\right)= \det \left(
\begin{array}{cc}
e^{\phi(t_n)}\cos(t_n)&e^{\phi(t_n)}\sin(t_n) \\
e^{\phi(\tau)}\cos(\tau)&e^{\phi(\tau)}\sin(\tau)  \\
\end{array}
\right)=\Bigl(-e^{\phi(t_n)} \Bigr) e^{\phi(\tau)}\sin(t_n-\tau),
$$
i.~e. $\eta_n\in S(t_n)$. According to the inequality (\ref{eq3})
we get $\phi(t_n)-\phi(\eta_n)> q$ for $n$. Since $\eta_n \to
\tau_0$ and $t_n \to t_0$ as $n \to \infty$, then passing to
limits in this inequality, we obtain $\phi(t_0)-\phi(\tau_0)\geq
q$, q.~e.~d.
\end{proof}

Let us fix a number $b^*\in(a,b)$. Now we obtain one remarkable
property of the function $\alpha(t)$ on the interval $(a,b^*]$.

\begin{cl}\label{cl5}
For every $t\in (a,b^*) \cap Sm$ at least one of the
following two assertions holds:

1) $t+\alpha(t)\geq b^*$;

2) there is $\xi=\xi(t)\in (t,t+\alpha(t))\cap Sm$ such that
$\alpha(t)> q\cdot \alpha(\xi)$.
\end{cl}

\begin{proof}
Suppose that Assertion 1) does not hold, i.~e. $t+\alpha(t) <
b^*$. Set $s=t+\alpha(t)$, then $t<s<b^*$. According to Claim
\ref{cl4}, $\phi(s)-\phi(t)\geq q
>1$. According to Claim \ref{cl3}, the function $\phi$ is differentiable
almost everywhere on the interval $[t,s]$, the derivative
$\phi^{\prime}$ is summable on this interval, and the inequality
$$
1<q \leq \phi(s)-\phi(t) \leq \int\limits_t^s \phi^{\prime} (\tau)
d\tau
$$
holds.

Further, for some number $\xi \in (t,s)\cap Sm$ the
inequality $\int_t^s \phi^{\prime} (\tau) d\tau \leq
(s-t)\phi^{\prime} (\xi)$ holds. Indeed, the set $(t,s)\cap Sm$
is a set of full measure on the interval $(t,s)$. If for all points
$\xi$ of this set we have $\int_t^s \phi^{\prime} (\tau) d\tau
> (s-t)\phi^{\prime} (\xi)$, then we get a contradiction
by integrating this inequality with respect to $\xi$
on $(t,s)$. Therefore, the required point $\xi \in
(t,s)\cap Sm$ does exist (such points consist of a set with
positive measure), hence,
$$
1<q \leq \phi(s)-\phi(t) \leq \int\limits_t^s \phi^{\prime} (\tau) d\tau \leq
(s-t)\phi^{\prime} (\xi).
$$
It is clear that  $\phi^{\prime}(\xi)>0$, therefore,
$\alpha(\xi)=\arccot(\phi^{\prime}(\xi))\in (0,\pi/2)$. Further,
$$
\phi^{\prime}(\xi)= \cot(\alpha(\xi))=1/\tan(\alpha(\xi))<1/\alpha(\xi),
$$
because $\tan(x)>x$ for $x\in(0,\pi/2)$. Consequently, $q \leq
\phi^{\prime}(\xi) \cdot \alpha(t)<
\frac{\alpha(t)}{\alpha(\xi)}$, and Assertion 2) is proved.
\end{proof}

Now, consider the set
$$
S^*=Sm \cap(a,b^*].
$$
It has full measure on the interval
$(a,b^*]$. Later on we will need some properties of the function
$t\mapsto \alpha(t)$ on the set $S^*$.

\begin{cl}\label{cl6} At least one of the following assertions holds:

1) there are a point $t^* \in (a,b^*]$ and a sequence $\{t_n\}$,
$t_n \in S^*$, such that $\alpha(t_n) \to 0$ and $t_n \to t^*$ as
$n \to \infty$;

2) $a>-\infty$ and there is $c>0$ such that $\alpha(t) \geq c$ for
all $t\in S^*$.
\end{cl}

\begin{proof}
Suppose that Assertion 1) does not hold and prove Assertion 2).

Consider any $b_0\in(a,b)$ and let $c_1\geq 0$ be the greatest
lower bound of the function $t\mapsto \alpha(t)$ on the set $S^*
\cap [b_0,b]$. If $c_1=0$, then using the compactness of the
interval $[b_0,b]$, it is easy to find a sequence $\{t_n\}$, $t_n
\in S^* \cap [b_0,b)\subset (a,b)$, that tends to some $t^* \in
S^* \cap [b_0,b]$ and such that $\alpha(t_n) \to 0$ as $n \to
\infty$. But Assertion 1) does not hold and, consequently, we get
the inequality $c_1>0$.

Now, set $c=\min \{c_1, b^*-b_0\}>0$. Let us show that $\alpha(t)
\geq c$ for all $t\in S^*$. Suppose that the set
$$
S=\{ t\in S^*\,|\, \alpha(t) <c\}
$$
is non-empty. Obviously, $S\subset (a,b_0]$. Note that for all
$t\in S$ the inequality $t+\alpha(t) < b^*$ holds (otherwise
$\alpha(t) \geq b^*-t\geq b^*-b_0\geq c$), hence, by Claim
\ref{cl5} there exists $\xi=\xi(t)\in (t,t+\alpha(t))\cap Sm$ such
that $\alpha(\xi)<\alpha(t)/q <\alpha(t)<c$, in particular, $\xi
\in S$.

Now, choose some $t_1\in S$ and construct a sequence of points
$\{t_n\}$ from $S$ by the following method: if $t_i$ is defined,
then put $t_{i+1}=\xi(t_i)$. By construction $t_i<t_{i+1}$, and,
since $c>\alpha(t_i)\geq q\cdot
\alpha(\xi(t_i))=q\cdot\alpha(t_{i+1})> \alpha(t_{i+1})$, then
$t_{i+1} \in S$. Since the constructed sequence increases and is
bounded from above by the number $b_0$ ($S\subset (a,b_0]$), it
has a finite limit $t^* \in (a,b_0]$, and the inequality
$\alpha(t_i)\geq q\cdot\alpha(t_{i+1})$ implies $\alpha(t_n) \to
0$ as $n \to \infty$. Therefore, Assertion 1) holds that is
impossible by our assumptions. Therefore, $S=\emptyset$, i.~e.
$\alpha(t) \geq c$ for all $t\in (a,b]$.

If $a>-\infty$, then we get Assertion 2) from statement of the
claim. Hence, we consider now the case $a=-\infty$.

For all $i\geq 1$ define the numbers $b_i$ by the recurrent
formula $b_i=b_{i-1}-q^i\cdot c$ ($b_0$ has been chosen earlier).
Let us prove by induction that
$$
\alpha(t)\geq c \cdot q^i
$$
for all $t\in (-\infty,b_i]\cap Sm $. We have proved this
inequality for $i=0$. Assume that it is proved for all $i<k$ and
prove it for $i=k$.

Consider any $t\in (-\infty, b_k]\cap Sm $. If $t+\alpha(t)\geq
b_{k-1}$, then $\alpha(t)
>b_{k-1}-t\geq b_{k-1}-b_k=c \cdot q^k$. If $t+\alpha(t) \leq b_{k-1}$
(that contradicts to the inequality $t+\alpha(t)
>b>b_{k-1}$), then by Claim \ref{cl5} there is
$\xi=\xi(t)\in (t,t+\alpha(t))\cap Sm \subset (t, b_{k-1})$ such
that $\alpha(t)\geq q\cdot \alpha(\xi)$. Since $\xi <b_{k-1}$,
then $\alpha(\xi)\geq c \cdot q^{k-1}$ by the inductive
assumption. Therefore, $\alpha(t)\geq c \cdot q^k$ in this case
too.

Now, it suffices to note that the just proved inequality
$\alpha(t)\geq c \cdot q^i$ contradicts to the inequality
$\alpha(t)=\arccot(\phi^{\prime}(t))<\pi$. Actually, for rather
large $i$ the inequality $c\cdot q^i
>\pi$ holds. This contradiction completes the proof of the claim.
\end{proof}
\smallskip

Now, we are ready to finish the proof of Theorem \ref{main0}. As
we have proved, either Assertion 1), or Assertion 2) from
the statement of Claim \ref{cl6} holds, therefore, it suffices to get a
contradiction in both these cases.
\smallskip

Suppose that Assertion 1) holds, i.~e. there are a point $t^* \in
(a,b^*]$ and a sequence $\{t_n\}$, $t_n \in S^* \subset Sm$ such
that $\alpha(t_n) \to 0$ and $t_n \to t^*$ as $n \to \infty$. Put
$s_n=t_n+\alpha(t_n)$. By Claim \ref{cl4} for rather large $n$ the
inequality $\phi(s_n)-\phi(t_n)\geq q$ holds (since $s_n\to t^*
<b$ as $n \to \infty$). But it is impossible, since $t_n \to t^*$,
$s_n= t_n+\alpha(t_n)\to t^*$, and the function $t \mapsto
\phi(t)$ is continuous at the point $t^*$. Therefore, we have
proved the theorem in this case.
\smallskip

Now, suppose that Assertion 2) holds, i.~e. $a>-\infty$ and there
is $c>0$ such that $\alpha(t) \geq c$ for all $t\in S^*$. Since
$\alpha(t)=\arccot(\phi^{\prime}(t))$, we get the inequality
$\phi^{\prime}(t) \leq \cot(c)\in \mathbb{R}$ для всех $t\in S^*$.
According to Claim \ref{cl3} for every $\eta\in (a,b^*)$ the
derivative $\phi^{\prime}$ is summable on the interval $[\eta,b^*]
\subset (a,b)$ and satisfies the inequality
$$ \phi(b^*)-\phi(\eta)
\leq \int\limits_{\eta}^{b^*} \phi^{\prime} (\tau) d\tau \leq
\cot(c)(b^*-\eta).
$$
Tending $\eta$  to $a$, we get $\phi(b^*)-\phi(a) \leq
\cot(c)(b^*-a)\in \mathbb{R}$, but the latter is impossible
because of $\phi(a)=-\infty$. Consequently, we have proved the
theorem in this case too.

\section{Various consequences and connections with other results}

At first we will use the assertion of Theorem \ref{main0} for
parameterizations of some special type. Let us consider two
continuous functions $h,g:[a,b)\subset \mathbb{R} \rightarrow
\mathbb{R}$ and suppose that the function $h$ is increasing and is
not a constant in any neighborhood of the point $a$. For any $x\in
(a,b)$ we consider the set of numbers $\tau \in (a, x]$, that are
points of local extremum of the function
\begin{equation}\label{eq3.5}
t \mapsto \Bigl(g(x)-g(a)\Bigr)h(t)-\Bigl(h(x)-h(a)\Bigr)g(t).
\end{equation}
Let $\mu(x)$ be the supremum of such $\tau$. The following theorem
gives a non-trivial information on a behavior of $\mu(x)$ as $x
\to a$.

\begin{thm}\label{main3.5}
Suppose in addition that there exists a finite limit
$\lim\limits_{x\to a} \frac{g(x)-g(a)}{h(x)-h(a)}$, then the
following inequality holds:
\begin{equation}\label{coshy.1}
\varlimsup_{x\to a} \frac{h(\mu(x))-h(a)}{h(x)-h(a)} \geq
\frac{1}{e}\,.
\end{equation}
\end{thm}

\begin{proof}
We use the assertion of Theorem \ref{main0} for the curve
$\gamma(t)=(g(t),h(t))\in \mathbb{R}^2$. For this curve, it is
clear that $\tau \in S(x)$ if and only if $\tau$ is a point of
extremum of the function (\ref{eq3.5}),
$D(x)=\sqrt{\bigl(g(x)-g(a)\bigr)^2+\bigl(h(x)-h(a)\bigr)^2}$.
Since the limit $P:=\lim\limits_{x\to a}
\frac{g(x)-g(a)}{h(x)-h(a)}$ exists and is finite, then it is easy
to see that
$$
DS(x)= \sup \{ D(\tau)\,|\, \tau \in S(x)\} \sim D(\mu(x))=
$$
$$
\sqrt{\bigl(g(\mu(x))-g(a)\bigr)^2+\bigl(h(\mu(x))-h(a)\bigr)^2}\,
\mbox{ as }\, x\to a.
$$
Set $L(x)=\frac{h(\mu(x))-h(a)}{h(x)-h(a)}$, then taking into
account the above asymptotic equality, we get
\begin{equation}\label{odno.1}
\left(\frac{DS(x)}{D(x)L(x)}\right)^2 \sim \frac{1+\left(
\frac{g(\mu(x))-g(a)}{h(\mu(x))-h(a)}\right)^2}{1+\left(
\frac{g(x)-g(a)}{h(x)-h(a)}\right)^2} \to \frac{1+P^2}{1+P^2}=1
\end{equation}
as $x\to a$. There, by Theorem  \ref{main0}
$$
\varlimsup_{x\to a}
\frac{h(\mu(x))-h(a)}{h(x)-h(a)}=\varlimsup_{x\to a}
\frac{DS(x)}{D(x)} \geq \frac{1}{e}\,,
$$
q. e. d.
\end{proof}

Now, suppose in addition that the function $h,g:[a,b)\subset \mathbb{R}
\rightarrow \mathbb{R}$ have derivatives and $h^{\prime}(t)>0$
on the interval $(a,b)$. Then by Cauchy's mean value theorem,
for any $x\in (a,b)$ there is $\tau \in (a, x)$ with the property:
\begin{equation}\label{eq4}
\frac{g(x)-g(a)}{h(x)-h(a)}=\frac{g^{\prime}(\tau)}{h^{\prime} (\tau)}.
\end{equation}
Let $\xi(x)$ be the supremum of such $\tau$. Obviously, any point
$\tau$ of extremum of the function (\ref{eq3.5}) satisfies the
equality (\ref{eq4}). Hence, $\xi(x) \geq \mu(x)$ for all $x$, and
Theorem \ref{main3.5} implies a non-trivial information on a
behavior of $\xi(x)$ as $x \to a$.

\begin{thm}\label{main4}
Suppose in addition that there exists a finite limit
$\lim\limits_{x\to a} \frac{g(x)-g(a)}{h(x)-h(a)}$, then the
following inequality holds:
\begin{equation}\label{coshy}
\varlimsup_{x\to a} \frac{h(\xi(x))-h(a)}{h(x)-h(a)} \geq \frac{1}{e}\,.
\end{equation}
\end{thm}

\begin{rmk}\label{re1}
Note that for any fixed function $h$ the inequality (\ref{coshy})
(as well as the inequality (\ref{coshy.1})) is best
possible in general. To show this, set
$g(x)=(h(x)-h(a))^{1+\alpha}$, where $\alpha
>0$. Since $h$ is monotone, for all $x\in (a,b)$ there is a unique $\tau=\xi(x) \in (a,x)$
that satisfies the equation (\ref{eq4}). Simple calculations imply
$$
\frac{h(\xi(x))-h(a)}{h(x)-h(a)}=\left(\frac{1}{1+\alpha}\right)^{1/{\alpha}}
$$
for all $x\in (a,b)$. Therefore, $\lim\limits_{x\to a}
\frac{h(\xi(x))-h(a)}{h(x)-h(a)}=\left(\frac{1}{1+\alpha}\right)^{1/{\alpha}}$.
Note also that $\lim\limits_{\alpha\to 0}
\left(\frac{1}{1+\alpha}\right)^{1/{\alpha}}=e^{-1}$. The same
example and the last formula in the proof of Theorem \ref{main3.5}
imply also the unimprovability of the inequality (\ref{eq1}).
\end{rmk}

In the case $h(x)= x$ Theorem \ref{main4} implies an
assertion on the asymptotic of mean value points in Lagrange's
theorem.

\begin{cor}\label{cor1}
Let $g:[a,b) \subset \mathbb{R} \rightarrow \mathbb{R}$ be a
continuous function that is differentiable on the interval
$(a,b)$. For every $x\in (a,b)$ denote by $\xi(x)$ the supremum of
numbers $\tau \in (a,x]$ such that $g^{\prime}(\tau) \cdot (x-a) =
g(x)-g(a)$. If the function $g$ has (right hand) derivative at the
point $a$, then the inequality
\begin{equation}\label{lagr}
\varlimsup_{x\to a} \frac{\xi(x)-a}{x-a} \geq \frac{1}{e}
\end{equation}
holds.
\end{cor}

The inequality (\ref{lagr}) becomes an equality, for example, for
the function $g:[0,1)\rightarrow \mathbb{R}$, defined by the
equality $g(x)=-\int\limits_0^x \frac{dt}{\ln t}$. The conjecture
of validity of the above corollary has been stated (as well as some other conjectures)
by Professor V.K.~Ionin. In the case, when
the derivative $g^{\prime}=:f$ is continuous, this conjecture
could be reformulated in the integral form. Consider a continuous
function $f:[a,b] \rightarrow \mathbb{R}$. For any $x\in (a,b]$
there exists $\tau \in [a,x]$ such that
$$
\int\limits_a^xf(t)\,d t=(x-a)f(\tau)
$$
(this is a partial case of the integral mean value theorem). Such
$\tau$ is unique if $f$ is strictly decreases or strictly
increases. In general case we set
$$
\eta(x):=\max \{ \tau \in [0,x]\,|\,\int\limits_a^xf(t)\,d t=xf(\tau)\}.
$$
Then (this is equivalent to Corollary \ref{cor1}) the inequality
$$
\varlimsup_{x\to a} \frac{\eta(x)-a}{x-a} \geq \frac{1}{e}
$$
holds. The latter inequality was proved at first in the paper
\cite{MV1}, one can find various generalisations of this result in
more recent papers \cite{MV3, MV2, MV5, MV6, MV7}.

The Theorems \ref{main3.5} and \ref{main4} give us a non-trivial
information on a behavior of the functions $\mu(x)$ and $\xi(x)$
by estimating the asymptotic of $\frac{h(\mu(x))-h(a)}{h(x)-h(a)}$
and $\frac{h(\xi(x))-h(a)}{h(x)-h(a)}$ respectively.

However, it would be desirable to get analogues assertions for the
values $\frac{\mu(x)-a}{x-a}$ and $\frac{\xi(x)-a}{x-a}$ (in the
case of Lagrange's theorem $h(x)= x$ we have got the required
results, of course). The following results imply some results of
this kind.

\begin{deff}
For a function $f : [a,b) \to \mathbb{R}$ we denote by
$\overline{\lim\limits_{x\to a}} \ess f(x)$ the greatest lower
bound of numbers $t \in \mathbb{R}$ such that $f(x) \leq t$ almost
everywhere on some interval $[a,\delta] \subset [a,b)$ (essential
upper limit). By analogy, $\lim\limits_{\overline{x\to a}} \ess
f(x)$ means the least upper bound of numbers $t \in \mathbb{R}$
such that $f(x) \geq t$ almost everywhere on some interval
$[a,\delta] \subset [a,b)$ (essential lower limit).
\end{deff}

\begin{lem}\label{main5}
Let $h:[a,b) \rightarrow \mathbb{R}$ be an increasing function and
suppose that
$$ C:=\overline{\lim\limits_{x\to a}}\ess
\frac{h(x)-h(a)}{(x-a)h^{\prime}(x)} <\infty \,.
$$
Then for any number $q \in (0,1)$ we get the inequality
$$
\varlimsup_{x\to a} \frac{h(a+q(x-a))-h(a)}{h(x)-h(a)} \leq q^{1/C}\,.
$$
\end{lem}

\begin{proof} Let us fix some number $\varepsilon >0$. Decreasing (if necessary) the number
$b$, we may assume that for almost all $t\in (a, b]$ the
inequality
$$
\frac{h(t)-h(a)}{(t-a)h^{\prime}(t)} <C+\varepsilon
$$
holds, or, equivalently:
$$
\frac{h^{\prime}(t)}{h(t)-h(a)}>\frac{1}{C+\varepsilon}\cdot \frac{1}{t-a}
$$
Integrating the latter inequality from $t=x_q:=a+q(x-a)$ to $t=x$,
we get
$$
\frac{1}{C+\varepsilon}\cdot \ln \frac{1}{q}< \int\limits_{x_q}^{x}\frac{h^{\prime}(t)
dt}{h(t)-h(a)}\leq {\left. \ln \Bigl(h(t)-h(a)\Bigr) \right|}_{t=x_q}^{t=x}=\ln
\frac{h(x)-h(a)}{h(x_q)-h(a)}.
$$
Here we  used properties of the increasing function $x \mapsto \ln
\bigl(h(x)-h(a)\bigr)$ (its increment at the interval is not less
than the integral of its derivative on the same interval, cf.
Remark \ref{re2}). After simple transformations we get
$$
\frac{h(x_q)-h(a)}{h(x)-h(a)}\leq q^{\frac{1}{C+\varepsilon}} \quad \mbox{ и }\quad
\varlimsup_{x\to a}\frac{h(x_q)-h(a)}{h(x)-h(a)}\leq q^{\frac{1}{C+\varepsilon}}.
$$
Since $\varepsilon >0$ is arbitrary, then the lemma is proved.
\end{proof}

\begin{thm}\label{main6}
Let $C= \overline{\lim\limits_{x\to a}}\ess
\frac{h(x)-h(a)}{(x-a)h^{\prime}(x)}$ in the assumptions and
notations of Theorem~\ref{main3.5}. Then the inequality
$$
\varlimsup_{x\to a} \frac{\mu(x)-a}{x-a}\geq e^{-C}
$$
holds. If, in addition, the assumptions of Theorem \ref{main4} are
fulfilled, then the inequality
$$
\varlimsup_{x\to a} \frac{\xi(x)-a}{x-a} \geq e^{-C}
$$
holds too.
\end{thm}

\begin{proof}
Let us prove the first inequality. For $C= \infty$ all is clear.
Further consider the case $C<\infty$. Suppose that the theorem is
false. Choose some number $q$ between $\varlimsup\limits_{x\to a}
\frac{\mu(x)-a}{x-a}$ and $e^{-C}$ ($0<q<e^{-C}<1$, in
particular). Without loss of generality we may assume that the inequality
$$
\frac{\mu(x)-a}{x-a} <q
$$
(or, equivalently, the inequality $\mu(x)<x_q=a+q(x-a)$) holds for all
$x\in (a, b)$.
Since the function $h$ increases, then $h(\mu(x)) <h(a+q(x-a))$.
By Theorem \ref{main3.5} we get
$$
\varlimsup_{x\to a} \frac{h(a+q(x-a)))-h(a)}{h(x)-h(a)} \geq
\varlimsup_{x\to a} \frac{h(\mu(x))-h(a)}{h(x)-h(a)} \geq
\frac{1}{e}\,.
$$
Now, Lemma \ref{main5} implies $ e^{-1} \leq q^{1/C}$, i.~e.
$e^{-C}\leq q$, that contradicts to the choice of the number $q$.
This contradiction proves the first inequality of the theorem.

The second inequality obviously follows from the first one and the
fact that $\xi(x) \geq \mu(x)$ for all $x$ in the conditions of
Theorem \ref{main4}.
\end{proof}
\smallskip

Let $f:[0,1] \rightarrow \mathbb{R}$ be a continuous function, and
let $\varphi :[0,1] \rightarrow \mathbb{R}$ be a summable and
non-negative. Let us define the function $\eta : [0,1] \to
\mathbb{R}$ in the following way: for any $x \in (0,1]$, $\eta
(x)$ is the maximum of numbers $\tau \in (0,x]$ satisfied the
equation
$$
\int\limits^x_0 \varphi (t)f(t)dt=f(\tau )\int\limits^x_0 \varphi (t)dt
$$
(such numbers $\tau \in (0,x]$ do exist because of the integral
mean value theorem). Theorem~\ref{main6} implies

\begin{thm}[\cite{MV2,MV4}]\label{main7} Let
$C=\overline{\lim\limits_{x\to 0}} \ess    \Bigl( \int\limits^x_0
\varphi(t)dt \cdot \Bigl( x \varphi (x)\Bigr)^{-1}\Bigr)$, then in
the notations as above the inequality
\begin{equation}\label{intm}
\varlimsup_{x\to 0} \frac{\eta(x)}{x} \geq e^{-C}
\end{equation}
holds.
\end{thm}

\begin{proof}
We may assume that the function $t \mapsto \varphi(t)$ is not zero
almost everywhere in any neighborhood of the point $0$ (otherwise,
in such neighborhood the equality $\eta(x)=x$ holds, and all is
clear). By the same manner we may assume that for some
$\varepsilon >0$  the value $\varphi(t)$ is not zero almost everywhere on any interval
$[c,d]\subset (0,\varepsilon)$ of non-zero length (otherwise,
$C=\infty$, and nothing to prove).

Now, we define two functions $g,h :[0,1] \rightarrow \mathbb{R}$
by the formulas
$$ g(x)=\int\limits^x_0 \varphi(t) f(t) dt, \quad
h(x)=\int\limits^x_0 \varphi(t) dt.
$$
It is clear that the function $x\mapsto h(x)$ increases and is not
a constant in any neighborhood of $0$, $g(0)=h(0)=0$,
$\lim\limits_{x\to 0} \frac{g(x)-g(0)}{h(x)-h(0)}=f(0) \in
\mathbb{R}$. Therefore, we can apply the part of Theorem
\ref{main6}, dealing with the function $x \mapsto \mu(x)$, to
these two functions (in this case $a=0$, $b=1$). Now, it suffices
to verify the inequality
$$
\eta(x)\geq \mu(x)
$$
for all $x$ (sufficiently close to $0$). In turn, for this goal it
is enough to prove the following assertion for all $x$ sufficiently close to $0$:
Every $\tau \in (0,x]$ provided a local extremum to the function
$$
t \mapsto \Bigl(g(x)-g(0)\Bigr)h(t)-\Bigl(h(x)-h(0)\Bigr)g(t)=
$$
$$
\int\limits^x_0 \varphi(s) f(s) ds  \int\limits^t_0 \varphi(s) ds
-\int\limits^x_0 \varphi(s) ds  \int\limits^t_0 \varphi(s) f(s) ds
=: \Psi(t)\,,
$$
satisfies the equality $\int\limits^x_0 \varphi (s)f(s)ds=f(\tau
)\int\limits^x_0 \varphi (s)ds $.

Let us suppose the contrary.
Then by the integral mean value theorem we get
$$
\Psi(\tau+\Delta)-\Psi(\tau)=\int\limits^{\tau+\Delta}_{\tau}
\varphi(s) ds \left(\int\limits^x_0 \varphi(s) f(s) ds - f(\nu)
\int\limits^x_0 \varphi(s) ds \right),
$$
where $\nu$ is some number between ${\tau+\Delta}$ and ${\tau}$.
For sufficiently small $\Delta$ the sign of the expression in the
brackets coincides with the sign of the (non-zero by our
assumption!) expression $\int\limits^x_0 \varphi (s)f(s)ds-f(\tau
)\int\limits^x_0 \varphi (s)ds$. At the same time, the sign of
$\int\limits^{\tau+\Delta}_{\tau} \varphi(s) ds$ coincides with
the sign of $\Delta$ (at least for all $\tau \in (0,\varepsilon)$,
where $\varepsilon$ is the number discussed in the beginning
of the proof). This means that the point $\tau$ could not be a
point of local extremum of the function $t \mapsto \Psi(t)$. The
theorem is proved.
\end{proof}

Notice that one can find various generalizations and refinements
of the proved theorem in the papers \cite{MV2,MV4}.

\section{Open questions}

Note that there is another (but quite natural from a geometric
point of view) proximity estimation of points $\gamma(\tau)$, $\tau \in T(t)\, (S(t))$,
to a point $\gamma(t)$.
Suppose that the curve $\gamma: [a,b] \rightarrow \mathbb{E}^2$ is
continuous and rectifiable. Of course, this assumption essentially
narrows a class of curves under investigation. Let $L:[a,b] \rightarrow
\mathbb{R}$ be such that $L(t)$ is the length of an arc (of the
curve $\gamma$), corresponding to values of the parameter from the
interval $[a,t]$. It is clear that for every $t \in (a,b)$ the
inequality $0\leq L(\tau) \leq L(t)$ holds for all $\tau \in S(t)$
(for all $\tau \in T(t)$ in the case of differentiable curve),
therefore, $0\leq \sup \{L(\tau)\,|\, \tau \in S(t)\, \bigl(\tau
\in T(t)\bigr)\} \leq L(t)$. It is quite possible that the
following conjecture is true.

\begin{conj}\label{co2}
Let $\gamma: [a,b] \rightarrow \mathbb{E}^2$ be an arbitrary
continuous rectifiable parametric curve. Then the  inequality
$$ \varlimsup_{t\to a} \frac{\sup \{L(\tau)\,|\, \tau
\in S(t)\}}{L(t)}\geq \frac{1}{e} $$
holds.
\end{conj}

A version of this conjecture for
differentiable curves also has doubtless interest.

\begin{conj}\label{co3}
Let $\gamma: [a,b] \rightarrow \mathbb{E}^2$ be an arbitrary
continuous rectifiable parametric curve such that for every $t\in
(a,b)$ there is a non-zero derivative vector $\gamma^{\prime}(t)$.
then the inequality
$$ \varlimsup_{t\to a} \frac{\sup
\{L(\tau)\,|\, \tau \in T(t)\}}{L(t)}\geq \frac{1}{e}
$$
holds.
\end{conj}

Obviously, the latter conjecture is true for all curves with the
property
\begin{equation}\label{eq6}
L(t) \sim D(t) \mbox{  as  } t\to a.
\end{equation}
However, this asymptotic equality is not universal, hence, the
result of Theorem \ref{main1} is not sufficient for studying  of
Conjecture \ref{co3}, it demands some special approach.
Nevertheless, it is well known that in the case of a smooth curve
$\gamma:[a,b]\rightarrow \mathbb{E}^2$ (when a curve has a continuous
derivative vector $\gamma^{\prime}(t)$, $t\in[a,b]$), the relation
(\ref{eq6}) is fulfilled. Hence, we get the following corollary
from Theorem \ref{main1}.

\begin{cor}\label{co4}
Let $\gamma: [a,b] \rightarrow \mathbb{E}^2$ be an arbitrary
regular smooth parametric curve. Then the inequality
$$
\varlimsup_{t\to a} \frac{\sup \{L(\tau)\,|\, \tau \in
T(t)\}}{L(t)}\geq \frac{1}{e}
$$
holds.
\end{cor}

Finally, we note the following independent
interesting problem: to find a convenient criterion for fulfillment
of the asymptotic equality (\ref{eq6}).

\end{document}